\documentclass[12pt, reqno]{article}

\usepackage{fullpage}
\usepackage{enumerate}
\usepackage{setspace}

\usepackage{amsmath,amsfonts,amssymb,graphics,amsthm, comment}
\usepackage{hyperref}
\hypersetup{
    colorlinks=true,
    linkcolor=blue,
    citecolor=red,
    urlcolor=blue,
    pdfborder={0 0 0}
}    
\usepackage{turnthepage,lipsum}        

\usepackage{graphicx}
\usepackage{xcolor}
\usepackage{bbm}
\usepackage{wrapfig} 
\usepackage{xcolor}

\usepackage[font=sf, labelfont={sf,bf}, margin=1cm]{caption}

\usepackage[T1]{fontenc}
\usepackage{mathabx,graphicx}

\usepackage{cleveref}

  \crefname{theorem}{Theorem}{Theorems}
  \crefname{thm}{Theorem}{Theorems}
  \crefname{lemma}{Lemma}{Lemmas}
  \crefname{lem}{Lemma}{Lemmas}
  \crefname{remark}{Remark}{Remarks}
  \crefname{prop}{Proposition}{Propositions}
  \crefname{proposition}{Proposition}{Propositions}
\crefname{notation}{Notation}{Notations}
\crefname{claim}{Claim}{Claims}
  \crefname{defn}{Definition}{Definitions}
  \crefname{corollary}{Corollary}{Corollaries}
  \crefname{section}{Section}{Sections}
  \crefname{figure}{Figure}{Figures}
  \crefname{exercise}{Exercise}{Exercises}
    \crefname{assumption}{Assumption}{Assumptions}

\newtheorem{thm}{Theorem}

\newtheorem{lemma}[thm]{Lemma}

\newtheorem{proposition}[thm]{Proposition}

\theoremstyle{definition}
\newtheorem{remark}[thm]{Remark}

\def \cal {\circlearrowleft}

\def\cD{\mathcal{D}}
\def\cC{\mathcal{C}}

\def\cA{\mathcal{A}}

\def\P{\mathbb{P}}

\def\R{\mathbb{R}}
\def\Z{\mathbb{Z}}

\def\R{\mathbb{R}}
\def\L{\mathbb{L}}

\def\1{\mathbf{1}}

\def  \p- {p\textunderscore}

\usepackage{extarrows}

\date{}

\title{A short proof of the discontinuity of phase transition in the planar random-cluster model with $q>4$}
\date{\today}

\author{
	Gourab Ray
	\thanks{University of Victoria, Department of Mathematics, Victoria, BC, V8W 2Y2, Canada. Supported in part by NSERC 50311-57400 and University of Victoria start-up 10000-27458.}
	\and
	Yinon Spinka
	\thanks{University of British Columbia,
    Department of Mathematics,
 	Vancouver, BC, V6T 1Z2, Canada.}
}

\AtEndDocument{
  \bigskip
  \small
  \textit{Emails:} \texttt{gourabray@uvic.ca, yinon@math.ubc.ca} \par
}

\begin{document}

\maketitle
\abstract{The goal of this paper is to provide a short proof of the discontinuity of phase transition for the random-cluster model on the square lattice with parameter $q>4$. This result was recently shown in \cite{duminil2016discontinuity} via the so-called \emph{Bethe ansatz} for the six-vertex model. Our proof also exploits the connection to the six-vertex model, but does not rely on the Bethe ansatz. Our argument is soft and only uses very basic properties of the random-cluster model (for example, we do not need the Russo--Seymour--Welsh theory developed recently in \cite{duminil2017continuity}).}

\section{Main results}

The random-cluster model is a well-known and, by now, also a well-studied dependent percolation model. Suppose we are given a finite graph $G$ and two parameters $p \in (0,1)$ and $q>0$, called the edge weight and cluster weight. The random-cluster model is a probability measure on $\{0,1\}^{E(G)}$ which assigns to a configuration $\omega$ a probability proportional to
\[ p^{o(\omega)}(1-p)^{c(\omega)} q^{k(\omega)},\]
where $o(\omega)$ is the number of open edges (edges with value 1), $c(\omega) = |E(G)| - o(\omega)$ is the number of closed edges, and $k(\omega)$ is the number of \emph{vertex clusters} in $\omega$. 

In this paper, we are concerned with the random-cluster model on the square lattice with parameter $q \ge 1$. As this is an infinite graph, one must take an appropriate limit in the definition above. The random-cluster model satisfies a monotonicity (FKG) property that gives rise two natural limits called the \emph{free random-cluster measure} and the \emph{wired random-cluster measure}. We denote these measures by $\P^{\text{f}}_{p,q}$ and $\P^{\text{w}}_{p,q}$, respectively. Both measures are translation-invariant probability measures on $\{0,1\}^{E(\Z^2)}$ which are extremal in a certain sense.

An important quantity in this model is the probability that the origin belongs to an infinite cluster. Let $\theta^{\text{f}}(p,q)$ and $\theta^{\text{w}}(p,q)$ denote these probabilities under $\P^{\text{f}}_{p,q}$ and $\P^{\text{w}}_{p,q}$, respectively.
It is well-known that the random-cluster model undergoes a phase transition as $p$ varies in the sense that, for any $q \ge 1$ there exists a critical threshold $p_c=p_c(q) \in (0,1)$ such that $\theta^{\text{f}}(p,q)$ and $\theta^{\text{w}}(p,q)$ are both 0 for all $p<p_c$ and are both positive for all $p>p_c$. In fact, in the case of the square lattice it has been shown~\cite{DCB_selfdual} that the critical threshold is $p_c=\sqrt{q}/(1+\sqrt{q})$ and it follows from~\cite[Theorem~6.17]{grimmett2006random} that $\P^{\text{f}}_{p,q}=\P^{\text{w}}_{p,q}$ for every $p \neq p_c$.

The behavior at the critical parameter $p_c$ is of great interest. One important question is whether the phase transition is \emph{continuous} or \emph{discontinuous}, meaning here whether $\theta^{\text{w}}(p,q)$ (equivalently, $\theta^{\text{f}}(p,q)$) is continuous or discontinuous as a function of $p$. It is well-known~\cite[Theorem~5.16]{grimmett2006random} that this function is continuous at all points $p \neq p_c$.
The problem then boils down to understanding the behavior at criticality.
Baxter~\cite{baxter1978solvable} conjectured that, on the square lattice $\Z^2$, the phase transition is continuous for $1 \le q \le 4$ and is discontinuous for $q>4$. This conjecture was recently verified in a combination of two beautiful papers: the regime $1 \le q \le 4$ was handled in~\cite{duminil2017continuity} and the case $q>4$ in~\cite{duminil2016discontinuity}. 
The existing proof of discontinuity for $q>4$ relies on an analysis of the so-called \emph{Bethe ansatz} aimed at computing the eigenvalues of a transfer matrix for the six-vertex model.
In this paper, we provide a short probabilistic proof of this result.

\begin{thm}\label{thm:main2}
The random-cluster model on $\Z^2$ with parameter $q>4$ undergoes a discontinuous phase transition in the sense that $\theta^{\text{f}}(p_c,q)=0$ and $\theta^{\text{w}}(p_c,q)>0$.
\end{thm}

We mention that the proof in~\cite{duminil2016discontinuity} yields a precise expression for the exponential rate of decay of the probability (under the free random-cluster measure) that the origin is connected to a far away point. Our proof relies on softer arguments and does not directly yield any information on the above rate of decay. However, by the dichotomy result in~\cite[Theorem~3]{duminil2017continuity}, it follows that the decay rate must be exponential (though our proof does not rely on the Russo--Seymour--Welsh theory developed in~\cite{duminil2017continuity}).
In addition, the soft nature of our arguments provides hope that one can apply them to other planar graphs.

Let us also point out that partial results in various forms were obtained before the breakthrough of~\cite{duminil2016discontinuity}. For example, a proof for $q \ge 25.72$ was obtained in  \cite{LMR86,LMMSS91,KS82} using Pirogov--Sinai theory and entropy techniques. Later, Duminil--Copin provided a softer argument~\cite{DC_disc} which worked for $q \ge 256$. Duminil--Copin also mentions in \cite{DC_disc} that his approach should yield the result for $q \ge 82$.

We also mention that \cref{thm:main2} implies that the critical planar $q$-state Potts model with (integer) $q \ge 5$ also undergoes a discontinuous phase transition.

\bigskip

Our proof goes through the Baxter--Kelland--Wu (BKW) coupling~\cite{BKW76} between the planar random-cluster model and the \emph{six-vertex model}.
A configuration in the six-vertex model is an assignment of arrows to the edges of $\Z^2$ (i.e., an orientation of the edges) satisfying the \emph{ice rule}: at any vertex of $\Z^2$, there are precisely two incoming and two outgoing arrows (see \cref{fig:spin}). The six-vertex model is a measure on such configurations. In order to keep the introduction minimal, we do not define this measure here (see \cref{sec:six} for the definition), but rather isolate the property of the BKW coupling that we need for the proof of \cref{thm:main2}.

We will henceforth view the random-cluster model as living on the edges of the even sublattice of $(\Z^2)^*$ (the dual of $\Z^2$), which we denote by $\L$ (so $\L$ is a rotated square lattice which can be thought of as $\sqrt{2}e^{i \pi/4} (\Z^2+(1/2,1/2))$). We similarly let $\L^*$ denote the odd sublattice of $(\Z^2)^*$ (which is the dual of $\L$). We will refer to $\L$ as the primal lattice and to $\L^*$ as the dual lattice. For an edge $e \in E(\L)$, we denote its dual edge by $e^* \in E(\L^*)$.
For a percolation configuration $\omega \in \{0,1\}^{E(\L)}$, we define the dual configuration $\omega^* \in \{0,1\}^{E(\L^*)}$ by $\omega^*_{e^*} := 1-\omega_e$.
By duality of the random-cluster model (see e.g. \cite{grimmett2006random}), if $\omega \in \{0,1\}^{E(\L)}$ is sampled from the free critical random-cluster measure on $\L$, then its dual $\omega^*$ is the wired critical random-cluster measure on $\L^*$.

\begin{figure}
\centering
\includegraphics[scale = 1.5]{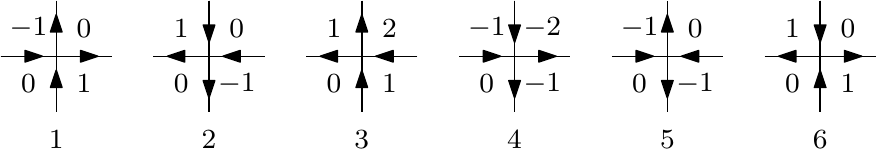}
\caption{The six possible types of arrow configurations satisfying the ice rule and the corresponding height function.}\label{fig:spin}
\end{figure}

Our proof of \cref{thm:main2} is a proof by contradiction.
Suppose that the random-cluster model undergoes a continuous phase transition for some $q>4$. A standard implication of this is that (and actually is equivalent to) the free and wired random-cluster measures coincide at criticality, i.e., $\P^{\text{f}}_{p_c,q}=\P^{\text{w}}_{p_c,q}$, and this measure assigns probability 0 to the event that there exists an infinite cluster in either the primal or dual configuration. Let $\omega$ be sampled from this measure. Consider the loops separating primal and dual clusters (which live on the medial lattice $\Z^2$). Independently orient each such loop randomly, either clockwise or anticlockwise with probabilities $e^{-\lambda} / (e^{\lambda}+e^{-\lambda})$ and $e^{\lambda} / (e^{\lambda}+e^{-\lambda})$, respectively. Since the loops are ``fully packed'' in the sense that every edge of $\Z^2$ is contained in some loop, by forgetting the loop structure (but not the orientation), we obtain a six-vertex configuration. Denote the law of this configuration by $\mu_{q,\lambda}$. See \cref{fig:fk-loops-height} for an illustration of these objects.

\begin{thm}\label{thm:c}
Let $q>4$ and suppose that the random-cluster model with parameter~$q$ undergoes a continuous phase transition. Let $\lambda$ be the unique positive solution to the equation $e^{\lambda} + e^{-\lambda} = \sqrt{q}$. Then $\mu_{q,\lambda} = \mu_{q,-\lambda}$.
\end{thm}

The proof of \cref{thm:c} relies on the observation that the BKW coupling is symmetric under the map $\lambda \mapsto -\lambda$ up to a boundary effect for the random-cluster model, which vanishes in the infinite-volume limit under the assumption of a continuous phase transition. We note that \cref{thm:c} remains true in the limiting case $q=4$ for the simple reason that $\lambda=0$. Indeed, when $q=4$ the phase transition is continuous and a six-vertex configuration sampled from $\mu_{4,0}$ is obtained by orienting the loops with no bias towards clockwise or anticlockwise orientation.

The proof of \cref{thm:c} is given in \cref{sec:BKW-and-proof}. Given \cref{thm:c}, the proof of \cref{thm:main2} is rather short and is given in the next section.

\medskip
\noindent\textbf{Acknowledgments.}
We are grateful to Alexander Glazman and Ron Peled who made available to us a draft of their paper~\cite{glazmanpeled2019}. We also thank them and Hugo Duminil-Copin for several useful comments on an earlier draft of this paper.

\section{Deducing Theorem~\ref{thm:main2} from Theorem~\ref{thm:c}}
\label{sec:thm-proof}

Fix $q > 4$. Throughout the proof, we assume that the phase transition in the random-cluster model with parameter $q$ is continuous in the sense that $\theta^{\text{w}}(p_c,q)=0$. It is standard~\cite[Theorem~5.33]{grimmett2006random} that this implies that the wired and free critical random-cluster measures coincide, i.e., $\P^{\text{w}}_{p_c,q}=\P^{\text{f}}_{p_c,q}$.
Our strategy is to show that a certain height function associated with a sample from $\mu_{q,\lambda}$ has a drift, which is positive if $\lambda>0$ and negative if $\lambda<0$. On the other hand, since \cref{thm:c} implies that $\mu_{q,\lambda}=\mu_{q,-\lambda}$ for some $\lambda>0$, this leads to a contradiction, showing that our assumption of continuity of phase transition must be false.

Fix $\lambda \neq 0$ and sample $\sigma$ from $\mu_{q,\lambda}$.
We start by defining the height function.
Recall that $\sigma$ satisfies the ice rule: at any vertex of $\Z^2$, there are two incoming and two outgoing edges.
We can think of the orientation of the edges of $\Z^2$ as a gradient of a function $h$ defined on the faces of $\Z^2$ (see \cref{fig:spin}). Precisely, fix $o \in \L$ and define $h \colon (\Z^2)^* \to \Z$ by setting $h(o):=0$ and defining $h(v)$ for any other $v \in (\Z^2)^*$ as follows. For every directed edge $e$ in $(\Z^2)^*$, associate a variable $H_e \in \{1,-1\}$ by letting $H_e=1$ if and only if $e$ crosses the oriented edge $e^*$ (where the orientation is determined by $\sigma$) from its left to its right.

Now let $\gamma$ be a path in $(\Z^2)^*$ from $o$ to $v$ and set $h(v) := \sum_{e \in \gamma} H_e$. To show that $h$ is well-defined, we must show that any two such paths yield the same value for $h(v)$. This follows from the fact that, by the ice rule, $H_{e_1}+H_{e_2}+H_{e_3}+H_{e_4}=0$ for every cycle $(e_1,e_2,e_3,e_4)$ in $(\Z^2)^*$. This shows that $h$ is a well-defined function of $\sigma$ and that this function satisfies $h(u)-h(v) \in \{1,-1\}$ for any adjacent $u,v \in (\Z^2)^*$.

Let us explain another way to obtain the height function.
Let $\omega$ be sampled from the unique critical random-cluster measure. Recall that $\sigma$ is obtained from $\omega$ by first randomly orienting the loops formed by interfaces between primal and dual clusters to obtain an oriented loop configuration $\vec\omega$, and then forgetting the loop structure (remembering only the arrows on the edges) to obtain $\sigma$. On the one hand, $h$ is defined through $\sigma$ as above. On the other hand, the same $h$ can also be defined directly through $\vec\omega$ as follows (see \cref{fig:fk-loops-height}). For every loop $L$, associate a variable $\xi_L \in \{1,-1\}$ by letting $\xi_L=1$ if and only if $L$ is oriented anticlockwise in $\omega^{\cal}$. Then the value of $h(v)$ is given by $\sum_L \xi_L - \sum_{L'} \xi_{L'}$, where the first sum is over the set of loops $L$ that surround $o$ but not $v$, and the second sum is over the set of loops $L'$ that surround $v$ but not $o$.

\begin{figure}
\centering
\includegraphics[scale = 0.5]{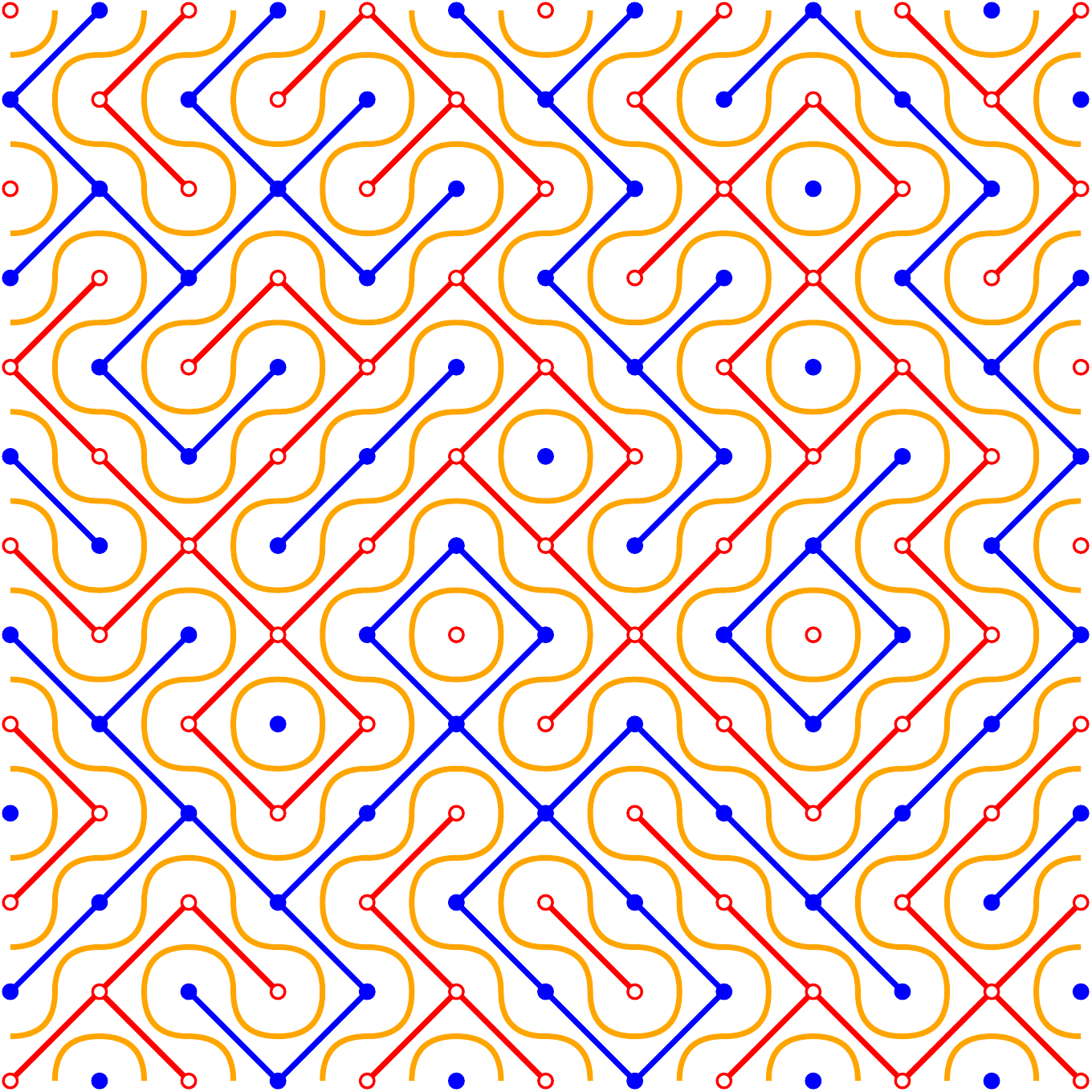}\qquad
\includegraphics[scale = 0.5]{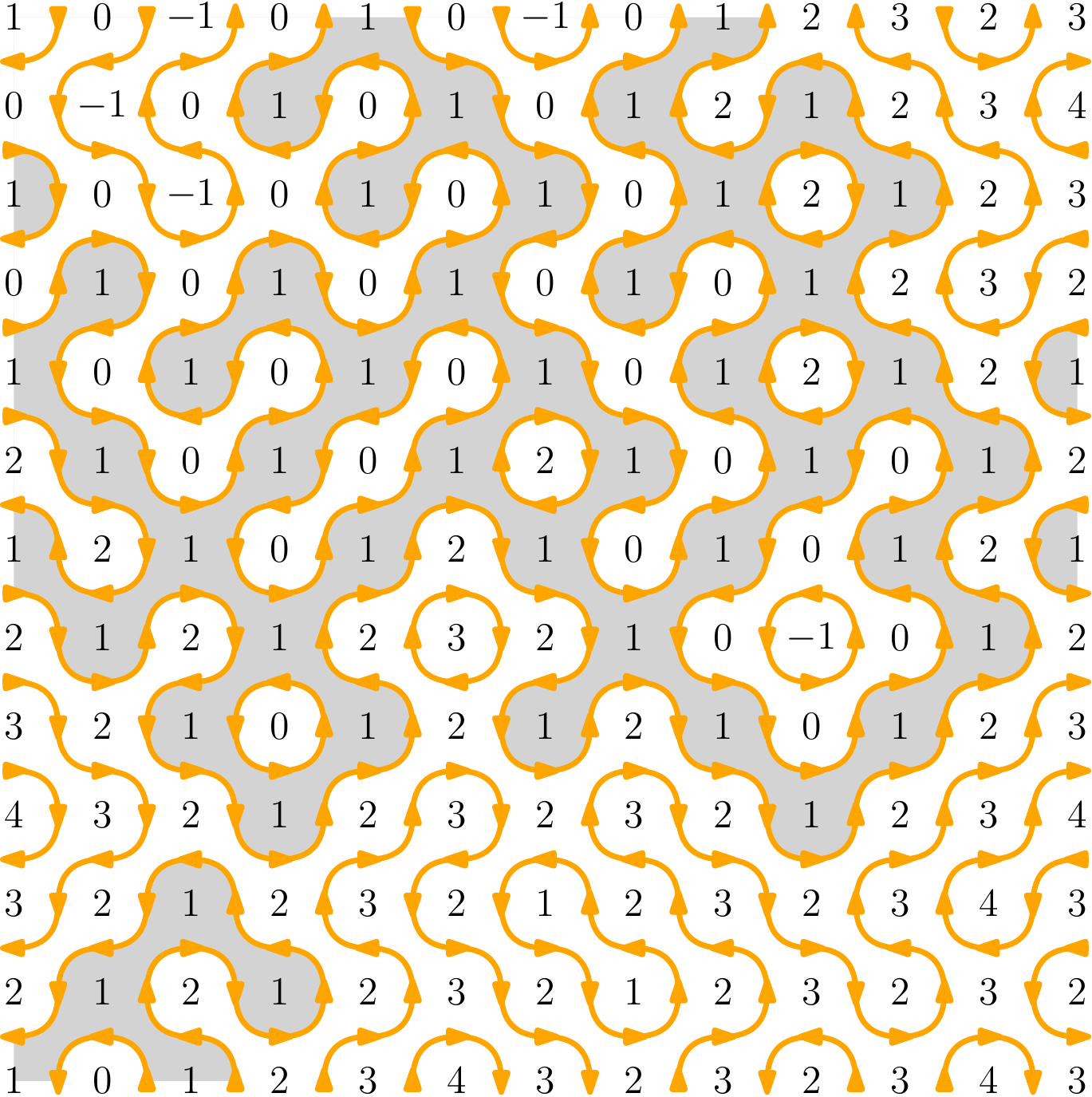}
\caption{\textit{Left:} A random-cluster configuration (in blue), its dual configuration (in red), and the loops (in orange) formed by the interfaces between primal and dual clusters. \textit{Right:} An oriented loop configuration and its associated height function. A height cluster is highlighted in gray.}\label{fig:fk-loops-height}
\end{figure}

The height function $h$ allows us to define maximal connected subsets with a given height. Call such a subset a \emph{height cluster}. Note that any such height cluster is either contained in $\L$ or in $\L^*$. In the former case, we call it a primal cluster and in the latter a dual cluster.
For a height cluster $\cA$ (or a subset of one), let $h(\cA)$ be its height.

We now define a partial order on finite connected subgraphs of $\mathbb L$ and $\mathbb L^*$ (later these will be height clusters or clusters in $\omega$) as follows.
Let $\cC \subset \L$ and $\cC' \subset \L^*$ be two such subgraphs.
We first note the following standard topological fact of planar duality (see~\cite[Proposition 11.2]{grimmett_perc}): there exists a unique cycle $\gamma(\cC)$ in $\L^*$ which surrounds $\cC$ and the dual of whose edges are contained in the edge-boundary of $\cC$ (edges which are not in $\cC$, but are incident to a vertex of $\cC$).
We write $\cC \prec \cC'$ if $\gamma(\cC) \subset \cC'$. We similarly define $\cC' \prec \cC$ if $\gamma(\cC') \subset \cC$.

Let us make a simple observation. Suppose that we are given two configurations $\tau \in \{0,1\}^{E(\L)}$ and $\tau' \in \{0,1\}^{E(\L^*)}$, both of which have only finite clusters, and such that either $e \in \tau$ or $e^* \in \tau'$ for every $e \in E(\L)$.
Let $\cC$ be a cluster in $\tau$ (resp. $\tau'$). Then there exists a unique cluster $\cC'$ in $\tau'$ (resp. $\tau$) such that $\cC \prec \cC'$. This follows by noting that $\gamma(\cC)$ and $\tau$ are disjoint, and thus, $\gamma(\cC)$ must be contained in $\tau'$.

Note that when all height clusters are finite, for every height cluster $\cA$ there is a unique height cluster $\cA'$ such that $\cA \prec \cA'$. Indeed, this is easily seen by considering the configurations $\tau = \{ \{u,v\} \in E(\L) : h(u)=h(v) \}$ and $\tau' = \{ \{u,v\} \in E(\L^*) : h(u)=h(v) \}$ and using the observation above (noting that for every edge $e$, either the two endpoints of $e$ have the same height, or the two endpoints of $e^*$ have the same height, so that $\tau$ and $\tau'$ indeed satisfy the required assumption).

Suppose that all height clusters are finite.
On this event, we can uniquely define a nested sequence \[ o \in \cA_0 \prec \cA_1 \prec \cA_2 \prec \dots \]
of alternating primal/dual height clusters.
Define the events
\begin{align*}
\cD^+ &:= \{ \text{all height clusters are finite and } h(\cA_n) \to +\infty\text{ as }n \to \infty \} ,\\
\cD^- &:= \{ \text{all height clusters are finite and } h(\cA_n) \to -\infty\text{ as }n \to \infty \} .
\end{align*}
We now show that $\mu_{q,\lambda}(\cD^+)=1$ for $\lambda>0$ and $\mu_{q,\lambda}(\cD^-)=1$ for $\lambda<0$. Since $\mu_{q,\lambda}=\mu_{q,-\lambda}$ for some $\lambda>0$ by \cref{thm:c}, this will result in a contradiction, thereby completing the proof.

Recall that $\omega$ is sampled from the unique critical random-cluster measure. Since, by assumption, all clusters in $\omega$ and $\omega^*$ are almost surely finite, we can uniquely define a nested sequence $o \in \cC_0 \prec \cC_1 \prec \dots$ of alternating primal and dual clusters in $\omega$ and $\omega^*$, and a sequence $L_1,L_2,\dots$ of loops such that $L_n$ is the interface between $\cC_{n-1}$ and $\cC_n$ (note that this loop is well-defined since $\gamma(\cC_{n-1}) \subset \cC_n$).
Note that the height on each cluster is constant (in other words, each $\cC_n$ is contained in some height cluster) so that $h(\cC_n)$ is well-defined.
Note also that $h(\cC_n)-h(\cC_{n-1}) = \xi_{L_n}$ for all $n \ge 1$ and that the loops $\{L_n\}_n$ are all distinct. Furthermore, the sequence $(L_n)_{n \ge 1}$ is measurable with respect to $\omega$.
Therefore, by definition of $\vec\omega$, we have that $\{\xi_{L_n}\}_{n \ge 1}$ are i.i.d.\ with $\P(L_1= \pm 1)=e^{\pm \lambda}/(e^{\lambda}+e^{-\lambda})$. Thus, $(h(\cC_n))_{n \ge 0}$ is a simple random walk with a positive (negative) drift when $\lambda$ is positive (negative). In particular, almost surely, $h(\cC_n) \to \infty$ if $\lambda>0$, and $h(\cC_n) \to -\infty$ if $\lambda<0$.

We now show that $\{ h(\cC_n) \to \infty\} \subset \cD^+$ and $\{ h(\cC_n) \to -\infty\} \subset \cD^-$, which will complete the proof.
We only show the former as the latter is analogous.
Suppose that $h(\cC_n) \to \infty$.
Note that all height clusters must be finite, since an infinite height cluster would necessarily intersect infinitely many $\cC_n$ (as these all surround $o$), and would therefore contradict the fact that $h(\cC_n) \to \infty$.
It remains to show that $h(\cA_n) \to \infty$.
To this end, it suffices to show that each $\cA_n$ intersects some $\cC_m$ (note that $m$ necessarily tends to infinity as $n \to \infty$).
For this, it suffices to show that $\bigcup_{k \ge 0} \cC_k$ contains an infinite simple path $\Gamma$ from $o$ in $\L \cup \L^*$ in the sense that consecutive vertices along the path are either adjacent in $\L$ or $\L^*$ or form an edge of $(\Z^2)^*$.
Indeed, this suffices because each $\cA_n$ surrounds $o$ (specifically, $\gamma(\cA_{n-1})$ is a cycle surrounding $o$) and therefore some vertex of $\cA_n$ must be at distance at most 1 from $\Gamma$.
To see that such a path $\Gamma$ exists, note that each $\cC_k$ is connected in $\L$ or $\L^*$, and $\cC_k$ can be connected to $\gamma(\cC_k) \subset \cC_{k+1}$ by an edge of $(\Z^2)^*$.
\qed

\section{The BKW coupling and proof of Theorem~\ref{thm:c}}
\label{sec:BKW-and-proof}

In \cref{sec:fk}, we introduce a modified random-cluster model with different weight for boundary clusters, and in \cref{sec:six}, we introduce the six-vertex model.
We then introduce the BKW coupling in \cref{sec:BKW}.
Finally, in \cref{sec:thm:c}, we give the proof of \cref{thm:c}.

\subsection{The random-cluster model}
\label{sec:fk}
The random-cluster model was introduced by Fortuin and Kasteleyn in 1970 \cite{F71,FK72} and is one of the most fundamental models in statistical physics. For background, we refer the reader to the excellent monographs due to Grimmett \cite{grimmett2006random} and Duminil--Copin \cite{duminil2017lectures}.

We consider here a modified version of the random-cluster model in which boundary clusters are assigned a different weight. This modification was introduced by Glazman and Peled~\cite{glazmanpeled2019} and will be useful for the BKW coupling in the following section. However, we quickly mention here that the use of this weighted model is not essential to our proof of \cref{thm:c} (see \cref{Rmk:torus}).

Recall our convention that the random-cluster model lives on the rotated square lattice $\L$.
The boundary of a set $\Lambda \subset \L$ is the subset of vertices in $\Lambda$ which are incident to some vertex outside $\Lambda$.
The \emph{weighted random-cluster measure} in a domain $\Lambda$ is a probability measure $\P^{\text{FK}}_{\Lambda,p,q,q_b}$ on $\{0,1\}^{E(\Lambda)}$ defined by
\begin{equation}
\P^{\text{FK}}_{\Lambda,p,q,q_b} (\omega) := \frac1{Z^{\text{FK}}_{\Lambda,p,q,q_b}}p^{o(\omega)}(1-p)^{c(\omega)} q^{k_i(\omega)} q_b^{k_b(\omega)}\label{eq:FK_weighted} ,
\end{equation}
where $o(\omega)$ is the number of open edges in $\omega$, $c(\omega)$ is the number of closed edges in $\omega$, $k_i(\omega)$ is the number of vertex clusters in $\omega$ which do not intersect the boundary, $k_b(\omega)$ is the number of vertex clusters in $\omega$ which intersect the boundary, and
$Z^{\text{FK}}_{\Lambda,p,q,q_b}$ is the partition function.

When $q_b = q$ we call this measure the free random-cluster measure and when $q_b = 1$ the wired random-cluster measure. These coincide with the standard definitions of the free and wired measures in the usual random-cluster model. We denote the free and wired measures by $\P^{\text{f}}_{\Lambda, p,q}$ and $\P^{\text{w}}_{\Lambda, p,q}$, respectively.

It is well-known that the random-cluster model satisfies certain monotonicity properties, namely, that the free measure increases and the wired measure decreases (in the sense of stochastic domination) as the domain $\Lambda$ increases. This allows us to take weak limits of $\P^{\text{w}}_{\Lambda, p,q}$ and $\P^{\text{f}}_{\Lambda, p,q}$ as $\Lambda \uparrow \L$, and we call the limiting measures $\P^{\text{w}}_{p,q}$ and $\P^{\text{f}}_{p,q}$, respectively.

The following monotonicity property is a special case of \cite[Corollary~5]{glazmanpeled2019}:
\begin{lemma}\label{lem:q_b}
For any $1 \le q_b \le q_b' \le q$, $\P^{\text{FK}}_{\Lambda,p,q,q_b}$ is stochastically dominated by $\P^{\text{FK}}_{\Lambda,p,q,q_b'}$.
\end{lemma}
\begin{proof}
Upon noting that, for any $e=\{u,v\} \in E(\Lambda)$ and $\tau \in \{0,1\}^{\Lambda \setminus \{e\}}$,
\[ \P^{\text{FK}}_{\Lambda,p,q,q_b}(\omega_e = 1 \mid \omega_{\Lambda \setminus \{e\}} = \tau) = \begin{cases}
 p &\text{if } u \overset{\tau}{\leftrightarrow} v \\
 \frac{p}{p+q_b(1-p)} &\text{if } u \overset{\tau}{\not\leftrightarrow} v,~ u \overset{\tau}\leftrightarrow \partial \Lambda,~ v \overset{\tau}{\leftrightarrow} \partial \Lambda \\
 \frac{p}{p+q(1-p)} &\text{otherwise}
\end{cases} ,\]
the lemma follows by an application of Holley's criterion.
\end{proof}

\subsection{The six-vertex model}\label{sec:six}

The six-vertex model is a model of arrow configurations on the edges of $\Z^2$ satisfying the \emph{ice rule}: at each vertex there are exactly two outgoing and two incoming arrows. This gives rise to one of six configurations (called \emph{types}) at each vertex, as depicted in \cref{fig:spin}.

Although this model is most natural with periodic boundary conditions, we define it in a finite domain with boundary, keeping in mind the parity of the boundary faces.
Recall that a vertex $v \in \L$ corresponds to a face of $\Z^2$. For a subset $\Lambda \subset \L$, let $\hat\Lambda$ denote the union of the four edges of this face over all $v \in \Lambda$.
An \emph{even domain} is a finite subset $\Lambda \subset \L$ such that the boundary of $\hat\Lambda$ consists of a simple cycle alternating between vertical and horizontal edges.
Note that a face of $\Z^2$ which is contained in $\hat\Lambda$ corresponds to some vertex in $\Lambda$.

Fix an even domain $\Lambda \subset \L$. Say that a vertex of $\hat\Lambda$ is \emph{internal} if all four of its incident edges lie in $\hat\Lambda$.
Let $\Omega^{6v, \cal}_{\Lambda}$ be the collection of all arrow configurations on the edges of $\hat\Lambda$ satisfying the ice rule at internal vertices, and where the arrows on the boundary edges are such that the boundary of $\hat\Lambda$ forms a directed cycle having anticlockwise orientation; see \cref{fig:oriented}.
Consider the probability measure on arrow configurations on the edges of $\hat\Lambda$ defined by
\begin{equation}
\P^{6v,\cal}_{\Lambda,c}(\omega_{6v}) = \frac1{Z^{6v,\cal}_{\Lambda,c}} \cdot c^{\# \{\text{type 5 or 6 vertices}\}} \cdot \1_{\Omega^{6v,\cal}_{\Lambda}}(\omega_{6v}), \label{eq:6v}
\end{equation}
where $Z^{6v,\cal}_{\Lambda,c}$ is the partition function.

Let us mention that this model is of wide independent interest. Following, Lieb~\cite{lieb1967exact,lieb2004exact,lleb2004residual}
and Yang--Yang~\cite{yang1966one} who
found an expression for the free
energy using the Bethe ansatz, the model is expected to exhibit different behavior when $0<c \le 2$ (the disordered phase) and when $c>2$ (the antiferroelectric
phase). A longstanding open question is to establish that the large scale behavior of the associated height function is close to a scalar multiple of the Gaussian free field for $c\in (0,2]$. We refer the interested reader to~\cite[Chapter~8]{baxter2016exactly} for more information about this model.

\subsection{The Baxter--Kelland--Wu coupling}\label{sec:BKW}

Temperley and Lieb~\cite{temperley1971relations} established a relation between the planar random-cluster model and the six-vertex model at the level of their partitions functions. Baxter, Kelland and Wu~\cite{BKW76} later gave a geometric description of this relation, which leads to a probabilistic coupling between the two models with parameters $q\ge 4$ and $c \ge 2$ (the relation also holds when $c<2$, but does not yield a probabilistic coupling). We describe a version of this coupling which closely follows~\cite{glazmanpeled2019}.

A useful way to represent the random-cluster model is by a nested loop configuration lying in $\Z^2$ (which is the medial lattice of $\L$). Take an even domain $\Lambda$ (recall the definitions of even and odd domains from \cref{sec:six}). For every random-cluster configuration in $\Lambda$, define its dual configuration to be $\omega_{e^*} = 1-\omega_e$, where $e^*$ is the dual edge of $e$. This defines interfaces between primal and dual clusters which form nested loops as illustrated in \cref{fig:oriented}. 

\begin{figure}
\centering
\includegraphics[scale = 0.9]{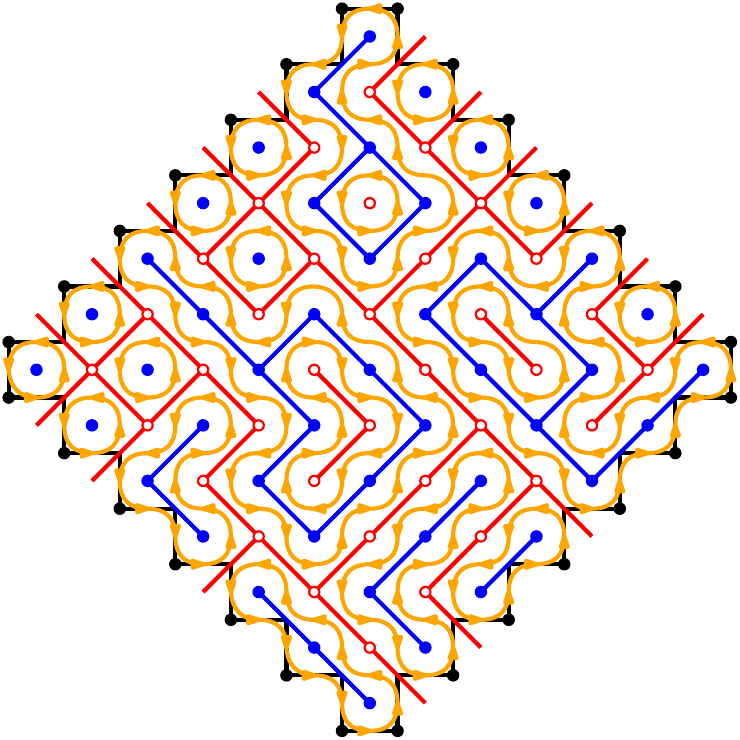}\qquad
\includegraphics[scale = 0.9]{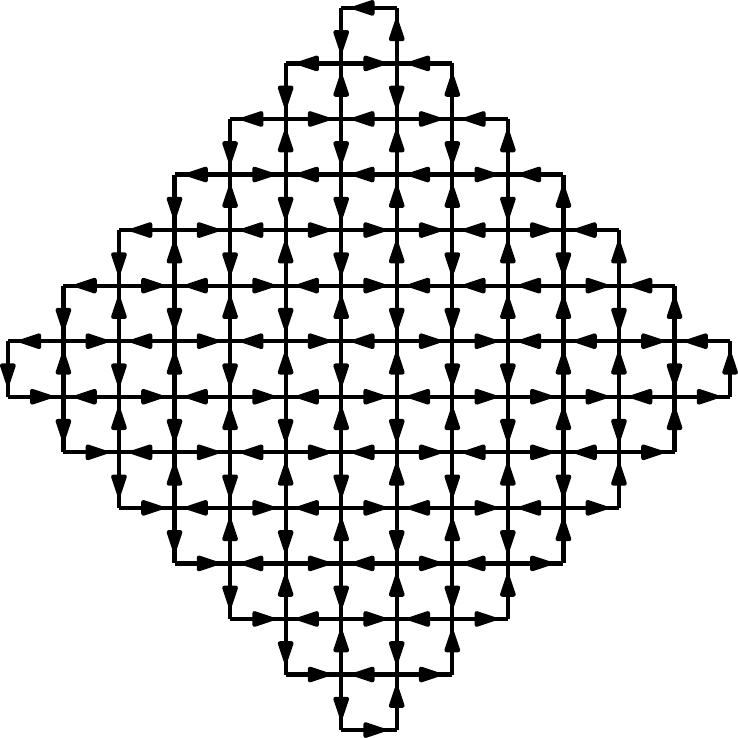}
\caption{\textit{Left:} The oriented random-cluster configuration with boundary loop given anticlockwise orientation. The orange loops are the interfaces between the primal (in blue) and dual (in red) clusters. Non-interior vertices are depicted by black bullets. \textit{Right:} The associated six-vertex configuration.}\label{fig:oriented}
\end{figure}

This loop representation allows us define the \emph{oriented random-cluster model} as follows. First sample a random-cluster configuration $\omega$ from $\P^{\text{FK}}_{\Lambda,p,q,q_b}$, and consider the resulting loop configuration  as above. We now fix $\lambda \in \R$ and orient the loops as follows.
Every loop containing a boundary edge is deterministically oriented in the anticlockwise direction. Every other loop is independently oriented in the clockwise or anticlockwise direction with probabilities $e^{-\lambda}/(e^\lambda+e^{-\lambda})$ and $e^{\lambda}/(e^\lambda+e^{-\lambda})$, respectively. 
This produces a random oriented loop configuration, whose law we denote by $\P^{\text{FK},\cal,\lambda}_{\Lambda,p,q,q_b}$.

We now describe an operation which will induce a measure preserving map between $\P^{6v,\cal}_{\Lambda,c}$ and $\P^{\text{FK},\cal,\lambda}_{\Lambda,p_c,q,q_b}$ (for well-chosen values of $c$, $q$, $q_b$ and $\lambda$). Given a sample $\omega_{6v}$ from $\P^{6v,\cal}_{\Lambda,c}$, for each internal vertex, we split the two incoming and two outgoing arrows into two non-crossing loop segments (each containing one incoming and one outgoing arrow) according to the rules shown in \cref{fig:split}. Note that a vertex of type 1 to 4 is split deterministically. On the other hand, a vertex of type 5 or 6 is split randomly in one of two possible ways with probabilities $e^{-\lambda/2}/(e^{-\lambda/2} + e^{\lambda/2})$ and $e^{\lambda/2}/(e^{-\lambda/2} + e^{\lambda/2})$, with the higher probability assigned to the case resulting in loop segments with anticlockwise orientations. These splits are carried out independently for each vertex. Call this (random) operation $\mathsf{Split}(\lambda)$.

\begin{figure}
\centering
\includegraphics[scale = 1.5]{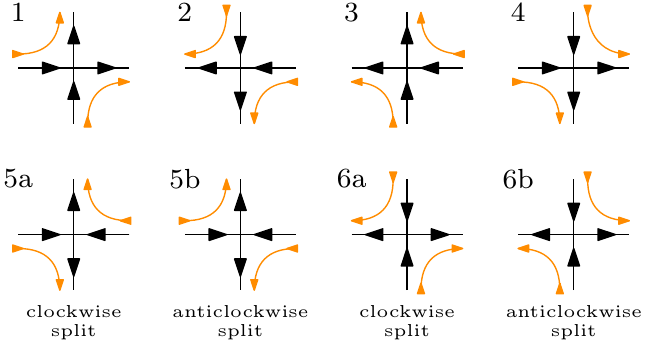}
\caption{The operation $\mathsf{Split}(\lambda)$. Types 1 to 4 can be split into two loop segments in a unique way. Types 5 and 6 can be split in two different ways and the choice of how to split is made randomly with the clockwise split having probability $e^{-\lambda/2}/(e^{\lambda/2} + e^{-\lambda/2})$ and the anticlockwise split having probability $e^{\lambda/2}/(e^{\lambda/2} + e^{-\lambda/2})$.}
\label{fig:split}
\end{figure}

\begin{proposition}\label{prop:split}
Suppose that
\begin{equation}
c = e^{-\lambda/2} + e^{\lambda/2}, \qquad \sqrt{q} = e^{-\lambda} + e^{\lambda}, \qquad p=p_c(q) =\frac{\sqrt{q}}{1+\sqrt{q}}.\label{eq:cq}
\end{equation}
Let $\Lambda$ be an even domain.
Then $\mathsf{Split}(\lambda)$ induces a measure preserving map from $\P^{6v,\cal}_{\Lambda,c}$ to $\P^{\text{FK},\cal,\lambda}_{\Lambda,p_c,q,e^{\lambda}\sqrt{q}}$.
\end{proposition}

Note that the existence of a solution to~\eqref{eq:cq} implies that $c \ge 2$, $q \ge 4$ and $c=\sqrt{2+\sqrt{q}}$. Moreover, given such $c$ and $q$, this equation uniquely defines $\lambda$ up to its sign.
That is, if $(c,q,\lambda)$ satisfies~\eqref{eq:cq} then so does $(c,q,-\lambda)$. Note also that $\mathsf{Split}^{-1}$ can be defined as the deterministic map that projects an oriented loop configuration onto its underlying six-vertex configuration. Thus, since the six-vertex measure $\P^{6v,\cal}_{\Lambda,c}$ does not depend explicitly on $\lambda$, the following is an immediate corollary to \cref{prop:split}.

\begin{proposition}\label{prop:split2}
Suppose that~\eqref{eq:cq} holds.
Let $\Lambda$ be an even domain. Then
\[ \mathsf{Split}^{-1} * \P^{\text{FK},\cal,\lambda}_{\Lambda,p_c,q,e^{\lambda}\sqrt{q}} = \P^{6v,\cal}_{\Lambda,c} = \mathsf{Split}^{-1} * \P^{\text{FK},\cal,-\lambda}_{\Lambda,p_c,q,e^{-\lambda}\sqrt{q}} .\]
\end{proposition}

\begin{proof}[Proof of \cref{prop:split}]
We first make some topological observations. It is straightforward that the operation $\mathsf{Split}(\lambda)$ creates an oriented loop configuration in $\hat\Lambda$, and because of the imposed boundary conditions, the loops intersecting the boundary are always oriented in the anticlockwise direction. The key idea due to Baxter, Kelland and Wu \cite{BKW76} is to measure the total angle turned by both loop segments at each vertex after applying $\mathsf{Split}(\lambda)$. Note that in types $1$ to~$4$, the total angle turned is 0: the two turns cancel each other out. However, in types 5 or 6, the total angle turned is either $\pi$ or $-\pi$, depending on how the arrows are split. Call the former an \emph{anticlockwise split} and the latter a \emph{clockwise split}. Note that the total winding of each loop is, on the one hand, the sum of angles it turns at each of its vertices, and, on the other hand, $\pm 2\pi$ according to its orientation. This leads to the following identity:
\begin{multline}
\pi (\# \text{(anticlockwise split)} - \# \text{(clockwise split)}  ) + \tfrac\pi 2 \cdot n_2(\Lambda) \\
= 2\pi (\# \text{(anticlockwise loops)} - \# \text{(clockwise loops)}),\label{eq:winding_identity}
\end{multline}
where $n_2(\Lambda)$ is the number of non-internal vertices in $\hat\Lambda$.
Here, $\# \text{(anticlockwise split)}$ and $\# \text{(clockwise split)}$ count splits in internal vertices, and $\# \text{(anticlockwise loops)}$ includes boundary loops in the count.

We now show that $\mathsf{Split}(\lambda)$ is measure preserving. Plugging in $p=p_c$ from \eqref{eq:cq} and $q_b=e^{\lambda}\sqrt{q}$ into \eqref{eq:FK_weighted}, we see that, for every $\omega \in \{0,1\}^{E(\Lambda)}$, 
\begin{equation}\label{eq:fk-weight}
\P^{\text{FK}}_{\Lambda,p_c,q,e^{\lambda }\sqrt {q}} (\omega) = \frac{1}{Z^{\text{FK}}_{\Lambda,p_c,q,e^{\lambda }\sqrt {q}}} \cdot \frac{(\sqrt{q})^{o(\omega)} q^{k_i(\omega)} (e^{\lambda}\sqrt{q})^{k_b(\omega)}  }{(1+\sqrt{q})^{|E(\Lambda)|}}.
\end{equation}
Let $\ell(\omega)$ be the number of loops in the loop configuration associated to $\omega$, let $\ell_i(\omega)$ the number of loops not containing a boundary edge of $\hat\Lambda$, and denote $\ell_b(\omega) = \ell(\omega) - \ell_i(\omega)$.
Let $k(\omega^*)$ be the number of clusters in the dual configuration $\omega^*$, where the outer face is thought of as a single vertex (so that clusters containing an dual boundary edge are identified and thus counted once).
We use the following identities:
\begin{equation}
k(\omega^*) + k(\omega) = \ell(\omega)+1, \quad k_b(\omega)=\ell_b(\omega) , \quad |V(\Lambda) |-o(\omega) + k(\omega^*) = k(\omega)+1. \label{eq:Euler}
\end{equation}
The first two equalities are easily verified.
The third equality is obtained by applying Euler's formula to $\omega$ and observing that the faces of $\omega$ (as a planar graph) are in correspondence with the dual clusters (recall that the cluster touching the outer face is wired).
It follows from~\eqref{eq:Euler} that $o(\omega)+2k_i(\omega)+k_b(\omega) = \ell_i(\omega) + |V(\Lambda)|$. Plugging this and $k_b(\omega)=\ell_b(\omega)$ back into \eqref{eq:fk-weight}, we see that
\begin{equation}
\P^{\text{FK}}_{\Lambda,p_c,q,e^{\lambda }\sqrt {q}} (\omega) = C \sqrt{q}^{\ell_i(\omega)} e^{\lambda \ell_b(\omega)} , \label{eq:unoriented_loop}
\end{equation}
where $$C = C(\Lambda) = \frac{1}{Z^{\text{FK}}_{\Lambda,p_c,q,e^{\lambda}\sqrt {q}}} \cdot  \frac{\sqrt{q}^{|V(\Lambda)|} }{(1+\sqrt{q})^{|E(\Lambda)|}}.$$
Thus, for any oriented loop configuration $\vec \omega$, we have that
\begin{equation}
\P^{\text{FK}, \cal}_{\Lambda,p_c,q,e^{\lambda }\sqrt {q}} (\vec \omega) = C e^{\lambda ( \# \text{(anticlockwise loops)} - \# \text{(clockwise loops)})}. \label{eq:oriented_FK}
\end{equation}
Recall that the boundary-touching loops always have anticlockwise orientation and we have used the second identity of \eqref{eq:Euler} to that end.

Finally, observe that the pushforward of $\P^{6v, \cal}_{\Lambda,c}$ by $\mathsf{Split}(\lambda)$ yields
\begin{align*}
\left(\mathsf{Split}(\lambda)*\P^{6v, \cal}_{\Lambda,c}\right)(\vec \omega)
 &= \P^{6v, \cal}_{\Lambda,c}(\omega_{6v}) \cdot \frac{e^{\frac\lambda 2 (\# \text{(anticlockwise split)} - \# \text{(clockwise split)})}}{c^{\# \{\text{type 5 or 6 vertices}\}}}\\
 &= \frac1{Z^{6v, \cal}_{\Lambda,c}} \cdot e^{\frac\lambda 2 (\# \text{(anticlockwise split)} - \# \text{(clockwise split)})},
\end{align*}
where $\omega_{6v} \in \Omega^{6v, \cal}_{\Lambda}$ is the unique six-vertex configuration which can lead to $\vec\omega$, and we count the number of splits when going from this $\omega_{6v}$ to $\vec\omega$.
Using \eqref{eq:winding_identity} and the fact that $n_2(\Lambda)$, $C$ and the partition functions depend only on $\Lambda$ and not on the configurations, we see that the above probability measure is the same as \eqref{eq:oriented_FK}.
 \end{proof}

\subsection{Proof of Theorem~\ref{thm:c}}
\label{sec:thm:c}

Fix $q>4$ and suppose that the random-cluster model undergoes a continuous phase transition in the sense that $\theta^{\text{w}}(p_c,q)=0$. It is standard~\cite[Theorem~5.33]{grimmett2006random} that this implies that the wired and free critical random-cluster measures coincide, i.e., $\P^{\text{w}}_{p_c,q}=\P^{\text{f}}_{p_c,q}$. Denote this measure by $\P^{\text{FK}}_{q}$.

Recall that $\P^{\text{FK}}_{\Lambda,p,q,1}$ and $\P^{\text{FK}}_{\Lambda,p,q,q}$ are the wired and free random-cluster measures in domain $\Lambda$ with parameters $p$ and $q$. Thus, by definition of $\P^{\text{w}}_{p,q}$ and $\P^{\text{f}}_{p,q}$, we have that $\P^{\text{FK}}_{\Lambda,p,q,1}$ and $\P^{\text{FK}}_{\Lambda,p,q,q}$ converge to $\P^{\text{w}}_{p,q}$ and $\P^{\text{f}}_{p,q}$, respectively (here and below, all limits are taken as $\Lambda$ increases to $\L$). In particular, $\P^{\text{FK}}_{\Lambda,p_c,q,1}$ and $\P^{\text{FK}}_{\Lambda,p_c,q,q}$ both converge to $\P^{\text{FK}}_{q}$.
Since \cref{lem:q_b} implies that, for any $1 \le q_b \le q$, $\P^{\text{FK}}_{\Lambda,p_c,q,q_b} $ is stochastically bounded between $\P^{\text{FK}}_{\Lambda,p_c,q,1}$ and $\P^{\text{FK}}_{\Lambda,p_c,q,q}$, it follows that $\P^{\text{FK}}_{\Lambda,p_c,q,q_b}$ converges to $\P^{\text{FK}}_{q}$ for any $1 \le q_b \le q$.

Let $\omega$ be sampled from $\P^{\text{FK}}_q$.
Consider the loops which form interfaces between primal and dual clusters in $\omega$.
Note that, almost surely, all such loops are finite.
Let $\P^{\text{FK},\cal,\lambda}_{q}$ denote the measure on oriented loops obtained by independently orienting each loop clockwise or anticlockwise with probabilities $e^{-\lambda}/(e^\lambda+e^{-\lambda})$ and $e^{\lambda}/(e^\lambda+e^{-\lambda})$, respectively.
Since all loops are finite and since $\P^{\text{FK}}_{\Lambda,p_c,q,q_b}$ converges to $\P^{\text{FK}}_{q}$ for any $1 \le q_b \le q$, it follows that $\P^{\text{FK},\cal,\lambda}_{\Lambda,p_c,q,q_b}$ converges to $\P^{\text{FK},\cal,\lambda}_{q}$ for any $1 \le q_b \le q$ and any $\lambda \in \R$.
Recall that $\mathsf{Split}^{-1}$ is the map that projects an oriented loop configuration onto its underlying six-vertex configuration, so that $\mathsf{Split}^{-1} * \P^{\text{FK},\cal,\lambda}_{q} = \mu_{q,\lambda}$ by definition.
In particular, we see that $\mathsf{Split}^{-1} * \P^{\text{FK},\cal,\lambda}_{\Lambda,p_c,q,e^{\lambda}\sqrt{q}}$ converges to $\mu_{q,\lambda}$ for any $\lambda$ such that $1 \le e^{\lambda} \sqrt{q} \le q$.

Now take $\lambda \neq 0$ to satisfy~\eqref{eq:cq}, and recall from \cref{prop:split2} that
\[ \mathsf{Split}^{-1} * \P^{\text{FK},\cal,\lambda}_{\Lambda,p_c,q,e^{\lambda}\sqrt{q}} = \P^{6v,\cal}_{\Lambda,c} = \mathsf{Split}^{-1} * \P^{\text{FK},\cal,-\lambda}_{\Lambda,p_c,q,e^{-\lambda}\sqrt{q}} .\]
Taking limits, we conclude that $\mu_{q,\lambda}=\mu_{q,-\lambda}$.
\qed

\begin{remark}\label{Rmk:torus}
We chose to work with the boundary-weighted random-cluster model defined in~\cref{sec:fk} as it allows for a simple comparison with the standard random-cluster model (namely, \cref{lem:q_b}). However, we feel that this choice is more of a convenience than a necessity. Indeed, we believe that the proof of \cref{thm:c} would work in a similar way if we were instead to consider the random-cluster model with periodic boundary conditions (i.e., on a torus). The primary difference would then be the presence of non-contractible loops, and one would need to show that these loops typically do not come near the origin.
\end{remark}

\bibliographystyle{amsplain}
\bibliography{library}

\end{document}